\documentclass[12pt]{article}

\usepackage{xcolor}
\usepackage{placeins}

\usepackage{setspace}

\usepackage{srcltx}

\usepackage{hyperref}

\usepackage{epsfig}

\usepackage{latexsym}

\usepackage{amssymb}

\usepackage{labelfig}

\newtheorem{lemma}{Lemma}[section]
\newtheorem{theorem}{Theorem}[section]
\newtheorem{proposition}{Proposition}[section]
\newtheorem{remark}{Remark}[section]
\newtheorem{definition}{Definition}[section]
\newenvironment{proof}{\noindent {\bf Proof.} \noindent}{\hfill \hbox{$\Box$} }

\begin{document}


\title{\Large Additive Schwarz preconditioner for the general finite volume element discretization of symmetric elliptic problems}

\author{\normalsize Leszek Marcinkowski\thanks{Faculty of Mathematics, University of Warsaw, Banacha 2, 02-097 Warszawa, Poland.
} \ \ \ \ Talal Rahman\thanks{Department of Computing, Mathematics and Physics, Bergen University College, Inndalsveien 28, 5020 Bergen, Norway}\ \ \ \ Atle Loneland\thanks{Department of Informatics, University of Bergen, Thorm{\o}hlensgt. 55, 5020 Bergen, Norway} \ \ \ \ Jan Valdman\thanks{Institute of Mathematics and Biomathematics, University of South Bohemia, Brani\v{s}ovsk\' {a} 31, 37005 \v{C}esk\'{e} Bud\v{e}jovice, and Institute of Information Theory and Automation of the ASCR, Pod Vod\'{a}renskou v\v{e}\v{z}\'{i} 4, 18208 Prague, Czech Republic.
}}

\date{}
\maketitle

\begin{abstract}
A symmetric and a nonsymmetric variant of the additive Schwarz pre\-conditioner are proposed for the solution of a
a general finite volume element discretization of symmetric elliptic problems, with large jumps in the entries of the coefficient matrices across subdomains. It is shown that the convergence of the pre\-conditioned GMRES iteration using the proposed preconditioners, depends poly\-logarithmically on the mesh parameters, in other words only weakly, and that they are robust with respect to the jumps in the coefficients.
\end{abstract}

\section{Introduction} \label{sec:intro}
The finite volume element method or the FVE method, also known in the literature as the  control volume finite element method or the CVFE, provides a systematic approach to construct a finite volume or a control volume discretization of the differential equations using a finite element approximation of the discrete solution. The method has drawn a lot of interest in the scientific communities because of its inheriting both the flexibility of using a finite element method and the conserving property of a finite volume discretization.

In this paper, we consider the classical finite volume discretization in which we seek for the discrete solution in the space of standard $P_1$ conforming finite element functions, i.e. continuous and piecewise linear functions, cf. \cite{Huang:1998:OFV,Ewing:2000:MFV,Ewing:2002:OTA}. Then we consider the second order elliptic partial differential equation with coefficients that may have large jumps across subdomains. Due to the finite volume discretization, the resulting systems are in general nonsymmetric, which become increasingly nonsymmetric for coefficients varying increasingly rapidly inside the finite elements. Designing robust and efficient algorithms for the numerical solution of such systems is often a challenge, particularly difficult is their analysis, which is not as well understood as it is for the symmetric system. The purpose of this paper is to design and study a class of robust and scalable preconditioners based on the additive Schwarz domain decomposition methodology, to be used in a preconditioned GMRES iteration for solving the system, cf. \cite{Saad:1986:GGM}.

Additive Schwarz methods have been extensively studied in the literature, cf. \cite{Toselli:2005:DDM}. When it comes to solving second order elliptic problems, the general focus has been in solving symmetric systems resulting from the finite element discretization of the problem. Despite the growing interest for finite volume elements, there exists only a limited number of research on fast methods for the nonsymmetric system resulting from the discretization, in particular methods like the domain decomposition which are considered among the most powerful methods for large scale computation have rarely been tested on finite volume elements. Among the few existing work known to the authors, are the works of \cite{Chou:2003:ADD,Zhang:2006:ODD} which consider overlapping variants of the additive Schwarz method (ASM) for the system. However, none of the existing work considers a substructuring type method. Such methods are called nonoverlapping Schwarz methods, and are known to have better convergence rate than their overlapping counter parts, cf. \cite{Toselli:2005:DDM}. The purpose of this work is to propose preconditioners for the finite volume element which are based on substructuring, and formulate them as additive Schwarz preconditioners. We show that their convergence depend poly-logarithmically on the mesh parameter.

For the general purpose of constructing an additive Schwarz preconditioner for a finite volume element discretization, and its analysis, we have in this paper formulated an abstract framework which is then further used for the preconditioners we are proposing. The framework borrows the basic ingredients of the abstract Schwarz framework for additive Schwarz methods, cf. \cite{Toselli:2005:DDM}, while the analysis follows the work of \cite{Cai:1992:DDA} where additive Schwarz methods were considered for the advection-diffusion problem. For further information on domain decomposition methods for nonsymmetric problems in general, we refer to \cite{Smith:1996:DDP,Toselli:2005:DDM,Mathew:2008:DDM}.

The paper is organized as follows: in Section~\ref{sec:diff-problem}, we present the differential problem,
and in Section~\ref{sec:FVD}, its finite volume element discretization. In Section~\ref{sec:ASM}, we present the two variants of the additive Schwarz preconditioners and the two main results, theorems~\ref{thm:ASM-edge-sym} and ~\ref{thm:ASM-edge-nsym}. The complete analysis is provided in the next two sections, the abstract framework in Section~\ref{sec:ASM-covol}, and the required estimates in Section~\ref{sec:techn-tools}.
Finally, numerical results are provided in Section~\ref{sec:numer-tests}.

Throughout this paper, we use the following notations: for any positive functions $w, x, y$, and $z$, and positive constants $c$ and $C$ independent of mesh parameters and jump coefficients: $x\lesssim y$ and $w\gtrsim z$ denote that $x\leq c y$ and $w\geq C z$, respectively.

%
%
%
\section{The differential problem}
\label{sec:diff-problem}
Given $\Omega$, a polygonal domain in the plane, and $f \in L^2(\Omega)$, the purpose is to solve the following differential equation,
\begin{eqnarray*}
- \nabla \cdot (A(x) \nabla u)(x)&=&f(x), \quad \ \ x\in \Omega, \\
       u(s)&=&0, \qquad \quad s \in \partial \Omega,
\end{eqnarray*}
where $A \in (L^\infty(\Omega))^4$ is a symmetric matrix valued function satisfying the uniform ellipticity  as follows,
$$
 \exists \, \alpha >0 \quad \mbox{such that} \quad \xi^T A(x)\xi\geq \alpha|\xi|_2^2 \quad \forall x \in \Omega \mbox{ and } \forall \xi \in \mathbb{R}^2,
$$
where $|\xi|_2^2=\xi_1^2+\xi_2^2$.
Further we consider $\alpha$  equal to 1 which can be always obtained by scaling the original problem by $\alpha^{-1}$.
We  assume that $\Omega$ is decomposed into a set of disjoint polygonal subdomains $\{D_j\}$ such that, in each subdomain $D_j$, $A(x)$ is continuous and smooth in the sense that
\begin{equation}\label{eq:subd_D_K}
 \|A\|_{W^{1,\infty}(D_i)}\leq C_\Omega,
\end{equation}
where $C_\Omega$ is a positive constant. We also assume that
$$
\exists \, \lambda_j >0 \quad \mbox{such that} \quad
\xi^T A(x)\xi \geq \lambda_j |\xi|_2^2 \geq  |\xi|_2^2 \quad \forall x \in D_j \mbox{ and } \forall \xi \in \mathbb{R}^2.
$$
Due to $A\in (L^\infty(D_j))^4$, we have the following,
$$
\exists\, \Lambda_j >0 \mbox{ such that } \quad | \nu^T A(x)\xi | \leq \Lambda_j| \nu|_2 |\xi|_2 \quad \forall x \in D_j  \mbox{ and } \forall \xi,\nu  \in \mathbb{R}^2.
$$
We also assume that $\Lambda_j\leq C_1 \lambda_j$ for a positive constant $C_1$. We then have,
\begin{equation}\label{eq:coeff-smooth-subd}
  \lambda_j |u|_{H^1(D_j)}^2\leq
      \int_{D_j} \nabla u^TA(x)\nabla u \:d x
    \leq \Lambda_j |u|_{H^1(D_j)}^2
\quad \forall u \in H^1(D_j).
\end{equation}
In the weak formulation, the differential problem is then to find $u\in H^1_0(\Omega)$ such that
\begin{equation}\label{eq:sym-bil-form}
 a(u,v)=f(v) \quad \forall v \in H^1_0(\Omega),
\end{equation}
where $$a(u,v) = \int_{\Omega} \nabla u^T A(x) \nabla u \:d x \quad \mbox{and}  \quad f(v)=\int_{\Omega}f v \:d x.$$

%
%
%
\section{The discrete problem}
\label{sec:FVD}
For the discretization of our problem, we use a finite volume element discretization, i.e. the equation (\ref{eq:sym-bil-form}) is discretized using the standard finite volume method on a mesh which is dual to the primal mesh, and the primal mesh is where the finite element space, our solution space, is defined, cf. \cite{Huang:1998:OFV,Ewing:2000:MFV,Ewing:2002:OTA}; for an overview of FV methods we refer to   \cite{Lin:2013:FVEM}.

Let $T_h=T_h(\Omega)$ be be a shape regular triangulation of $\Omega$, cf. \cite{Brenner:2002:MTF} or \cite{Braess:1997:FE}, hereon referred to as the primal mesh, consisting of triangles $\{\tau\}$ with the size parameter $h=\max_{\tau \in T_h} \mathrm{diam}(\tau)$, and let $\Omega_h$, $\partial\Omega_h$, and $\overline{\Omega}_h$ be the sets of triangle vertices corresponding to $\Omega$, $\partial\Omega$, and $\overline{\Omega}$, respectively. We assume that each $\tau\in T_h$ is contained in one of $D_j$.

Let $V_h$ be the conforming linear finite element space consisting of functions which are continuous piecewise linear over the triangulation $T_h$, and which are equal to zero on $\partial \Omega$.

Let $T_h^* = T_h^*(\Omega)$ be the dual mesh corresponding to $T_h$. For simplicity we use the so called Donald mesh for the dual mesh. For each triangle $\tau\in T^h$, let $c_{\tau}$ be the centroid, $x_j, \; j=1,2,3$ the three vertices, and $m_{k l}=m_{l k}, \; k,l=1,2,3$ the three edge midpoints. Divide each triangle $\tau$ into three polygonal regions inside the triangle by connecting its edge midpoints $m_{k l}=m_{l k}$ to its centroid $c_{\tau}$ with straight lines.
One such polygonal region $\omega_{\tau,x_1}\subset \tau$, associated with the vertex $x_1$, as illustrated in Figure~\ref{fig:cov-tr}, is the region which is enclosed by the line segments $\overline{c_{\tau}m_{13}}, \overline{m_{13}x_1}, \overline{x_1m_{12}}$, and $\overline{m_{12}c_{\tau}}$, and whose vertices are $c_{\tau}, m_{13}, x_1$, and $m_{12}$.
Now let $\omega_{x_k}$ be the control volume associated with the vertex $x_k$, which is the sum of all such polygonal regions associated with the vertex $x_k$, i.e.
$$
\omega_{x_k}=\bigcup_{\{\tau\in T_h: \: x_k \mbox{ is a vertex of } \tau\}} \omega_{\tau,x_k}
$$
The set of all such control volumes form our dual mesh, i.e.
$T_h^*=T_h^*(\Omega)=\{\omega_x\}_{x \in \overline{\Omega}_h}$.
A control volume $\omega_{x_k}$ is called a boundary control volume if $x_k\in \partial\Omega_h$.

\begin{figure}[htb]
 \centerline{
 \SetLabels
  \L\B (.5*.45) $c_{\tau}$ \\
  \L\B (1.*.1) $x_1$ \\
  \L\B (-.1*.05) $x_2$ \\
  \L\B (.52*1.) $x_3$ \\
  \L\B (.5*-.1) $m_{12}$ \\
  \L\B (.8*.5) $m_{13}$ \\
  \L\B (.1*.5) $m_{23}$ \\
 \endSetLabels
 \AffixLabels{
 \includegraphics[width=0.3\textwidth]{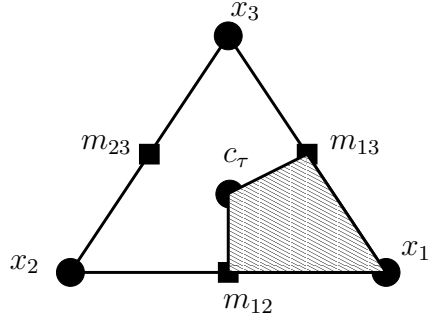}}}
 \caption{Showing $\omega_{\tau,x_1}$ (shaded region) which is part of the control volume $\omega_{x_1}$ restricted to the triangle $\tau$. The control volume is associated with the vertex $x_1$.} \label{fig:cov-tr}
\end{figure}

Let $V_h^*$ be the space of  piecewise constant functions over the dual mesh $T_h^*$, which have values equal to zero on $\partial \Omega_h$.
We let the nodal basis of $V_h$ be $\{\phi_x\}_{x\in \Omega_h}$, where $\phi_x$ is the standard finite element basis function which is equal to one at the vertex $x$ and zero at all other vertices. Analogously, the nodal basis of $V_h^*$ is $\{\psi_x\}_{x\in \Omega_h}$ where $\psi_x$ is a piecewise constant function which is equal to one over the control volume $\omega_x$ associated with the vertex $x$, and is zero elsewhere.

The two interpolatory operators, $I_h$ and $I_h^*$, are defined as follows. $I_h:C(\Omega)+V_h^*\rightarrow V_h$ and $I_h^*:C(\Omega)\rightarrow V_h^*$ are given respectively as
\begin{eqnarray*}
  I_h^* v=\sum_{x\in \Omega_h} v(x)\psi_x \quad \mbox{and} \quad
  I_h v=\sum_{x\in \Omega_h} v(x)\phi_x.
\end{eqnarray*}
We note here that $I_hI_h^*v=v$ for $v \in V_h$, as well as $I_h^* I_h u=u$ for $u \in V_h^*$.

Let the finite volume bilinear form be defined on $V_h\times V_h^*$ as $a_{FV}: V_h\times V_h^*\rightarrow \mathbb{R}$ such that
$$
a_{FV}(u,v)=-\sum_{x_i\in \Omega_h}v_i \int_{\partial V_i}A \nabla u \textbf{n} ds \quad u\in V_h, \; v\in V_h^*,
$$
or equivalently on $V_h\times V_h$ as $a_h: V_h\times V_h\rightarrow \mathbb{R}$ such that
\begin{eqnarray}\label{eq:FV-bil-form}
  a_h(u,v)&=&a_{FV}(u,I_h^* v)
\end{eqnarray}
for $u,v\in V_h$. We note here that $a_h(\cdot,\cdot)$ is a nonsymmetric bilinear form in general, while $a(\cdot,\cdot)$ is a symmetric bilinear form.
The discrete problem is then to find $u_h \in V_h$ such that
\begin{equation} \label{eq:FVdp}
  a_{FV}(u_h,v)=f(v)  \quad \forall v \in V_h^*,
\end{equation}
or equivalently
$$
  a_h(u_h,v)=f(I_h^* v) \quad \forall v \in V_h .
$$
The problem has a unique solution if $h$ is sufficiently small which can be shown following the lines of \cite{Ewing:2002:OTA}.

We close this section with the following remark. Note that in some cases, the bilinear form $a_h(\cdot,\cdot)$ may equal the symmetric bilinear form $a(\cdot,\cdot)$, as for instance in the case when the matrix $A$ is piecewise constant over the subdomains $\{D_j\}$. This may not be true if we choose to use a different dual mesh or if $A$ is not piecewise constant. In this paper, we only consider the case when the bilinear form $a_h(\cdot,\cdot)$ is nonsymmetric.

\section{An edge based ASM method}
\label{sec:ASM}
In this section, we propose an edge based Additive Schwarz method (ASM) for
the finite volume element discretization described in Section~\ref{sec:FVD}.

We assume that we have a partition of $\Omega$ into the set of $N$ polygonal subdomains $\{\Omega_k\}_{k=1}^N$, such that they form a coarse triangulation or a coarse mesh of $\Omega$, which is shape regular in the sense of \cite{Brenner:1999:CNS}. Let $H_k=\mathrm{diam}(\Omega_k)$. We assume that each subdomain $\Omega_k$ lies in exactly one of the polygonal subdomains $\{D_j\}$ described earlier, and that none of their boundaries cross each other.
The interface
$$
\Gamma=\bigcup_{k=1}^N \partial\Omega_k\setminus \partial\Omega,
$$
which is the sum of all subdomain edges and subdomain vertices or crosspoints (not lying on the boundary $\partial\Omega$), plays a crucial role in the design of our preconditioner. We also assume that the primal mesh $T_h$ is perfectly aligned with the partitioning of $\Omega$, in other words, no edges of the primal mesh cross any edge of the coarse mesh. As a consequence, the coefficient matrix $A(x)$ restricted to a subdomain $\Omega_k$ is in $(W^{1,\infty}(\Omega_k))^4$, and hence (cf. (\ref{eq:subd_D_K}))
$$
    \|A\|_{W^{1,\infty}(\Omega_k)}\leq  \|A\|_{W^{1,\infty}(D_j)}\leq C_\Omega.
$$

Each subdomain $\Omega_k$ inherits its own local triangulation from the $T_h$, denote it by $T_h(\Omega_k)=\{\tau\in T_h: \tau\subset \Omega_k \}$.
Let $V_h(\Omega_k)$ be the space of continuous and piecewise linear functions over the triangulation $T_h(\Omega_k)$, which are zero on $\partial\Omega\cap\partial\Omega_k$, and let $V_{h,0}(\Omega_k):=V_h(\Omega_k)\cap H^1_0(\Omega_k)$. The local spaces are equipped with the bilinear form
$$a_k(u,v)=\int_{\Omega_k } \nabla u^T A(x) \nabla v\:dx.$$
We define the local projection operator  $\mathcal{P}_k: V_h(\Omega_k)\rightarrow V_{h,0}(\Omega_k)$ such that
$$
  a_k(\mathcal{P}_ku,v)=a_k(u,v)\qquad \forall v\in V_{h,0}(\Omega_k)
$$
and the local discrete harmonic extension operator  $\mathcal{H}_k: V_h(\Omega_k)\rightarrow V_h(\Omega_k)$ such that
$$
  \mathcal{H}_k u =u- \mathcal{P}_k u.
$$
Note that $\mathcal{H}_ku$ is equal to $u$ on the boundary $\partial\Omega_k$, and discrete harmonic inside $\Omega_k$ in the sense that
$$
 a_k(\mathcal{H}_ku, v)=0 \quad \forall v\in V_{h,0}(\Omega_k).
$$
The local and global spaces of discrete harmonic functions are then defined as
$$ W_k=\mathcal{H}_k V_h(\Omega_k) \quad \mbox{and} \quad
W=\mathcal{H}V_h=\{u \in V_h: u_{|\Omega_k}=\mathcal{H}_k u_{|\Omega_k}\},
$$
respectively.

We now define the subspaces required for the ASM preconditioner, cf. \cite{Smith:1996:DDP,Toselli:2005:DDM,Mathew:2008:DDM}. For each subdomain $\Omega_k$, the local subspace $V_k\subset V_h$ is defined as $V_{h,0}(\Omega_k)$ extending it by zero to the rest of the subdomains, i.e.
$$
  V_k = \{v\in V_h: v_{|\overline{\Omega}_k}\in V_{h,0}(\Omega_k) \quad
             \mbox{and} \quad v_{|\Omega\setminus \overline{\Omega}_k}=0 \}
$$

The coarse space $V_0\subset W$ is defined as the space of discrete harmonic functions which are piecewise linear over the subdomain edges.
The dimension of $V_0$ equals the cardinality of $\mathcal{V}=\bigcup_k \mathcal{V}_k$, where  $\mathcal{V}_k$ is the set of all subdomain vertices which are not on the boundary $\partial\Omega$, in other words its dimension is the number of crosspoints.

Finally, the local edge based subspaces which are defined as follows.
For each subdomain edge $\Gamma_{k l}$, which is the interface between $\Omega_k$ and $\Omega_l$, we let $V_{k l}\subset W$ be the local edge based subspace consisting of functions which may be nonzero inside $\Gamma_{k l}$, but zero on the rest of the interface $\Gamma$, and discrete harmonic in the subdomains. It is not difficult to see that the support of $V_{k l}$ is contained in $\overline{\Omega}_k \cup \overline{\Omega}_l$.

We have the following decompositions of the finite element spaces $W$ and $V_h$. Both the symmetric and the nonsymmetric variant of the preconditioner use the same decompositions:
\begin{eqnarray}
\nonumber
W&=&V_0+\sum_{\Gamma_{k l}\subset \Gamma} V_{k l},\\
\label{eq:space-decomp}
V_h&=& W+ \sum_{k=1}^N V_k
   =V_0+\sum_{\Gamma_{k l}\subset \Gamma} V_{k l} +  \sum_{k=1}^N V_k.
\end{eqnarray}
Note that the subspaces of $W$ are a-orthogonal to the subspaces $V_k, \; k=1,\ldots,N$.

\subsection{Symmetric preconditioner}
For the symmetric variant of the preconditioner, we define the coarse and the local operators
$T_k:V_h\rightarrow V_k$ for $k=0,1,\ldots,N$, as
$$
  a(T_ku,v)=a_h(u,v) \qquad \forall v\in V_k,
$$
and the local edge operators
$T_{k l}:V_h\rightarrow V_{k l}$ for all $\Gamma_{k l} \subset \Gamma$, as
$$
  a(T_{k l}u,v)=a_h(u,v) \qquad \forall v\in V_{k l}.
$$
Now, defining the additive Schwarz operator $T$ as
$$
 T=T_0 + \sum_{k=1}^N T_k + \sum_{\Gamma_{k l}\subset \Gamma} T_{k l}
$$
we can replace the variational equation (\ref{eq:FVdp}) by the equivalent system of equations
\begin{equation} \label{eq:ndp}
 T u_h=g
\end{equation}
in the operator form, where $g=g_0 + \sum_{k=1}^N g_k + \sum_{\Gamma_{k l}\subset \Gamma} g_{k l}$, $g_k=T_ku_h^*$ for $k=0,1,\ldots,N$, and $g_{k l} = T_{k l}u_h^*$ for $\Gamma_{k l}\subset\Gamma$, with $u_h^*$ being the exact solution of (\ref{eq:ndp}).

\begin{theorem}\label{thm:ASM-edge-sym}
There exists an $h_1$ such that, if $h\leq h_1$, then for any $u\in V_h$
$$
  a(Tu,Tu)\lesssim a(u,u),  \qquad
  a(Tu,u)\gtrsim \left(1+\log\left(\frac{H}{h}\right)\right)^{-2} a(u,u),
$$
where $H=\max_k(H_k)$ and $H_k=\mathrm{diam}(\Omega_k)$.
\end{theorem}
The theorem is proved using the abstract results of Section~\ref{sec:ASM-covol}, e.g. Theorem~\ref{thm:ASM-covol-symm}, and the propositions~\ref{prop:nonsym-part-h-bound},~\ref{C-S-spec-rad-bnd}, and~\ref{prop:low-bnd-sym} from Section~\ref{sec:techn-tools} below.

Then Theorem~\ref{thm:GMRES-estimate} gives an estimate of the convergence speed of the GMRES method applied for solving (\ref{eq:ndp}).

\subsection{Nonsymmetric preconditioner}
For the nonsymmetric variant of the preconditioner, we define a new set of coarse and local operators,
$S_k:V_h\rightarrow V_k$ for $k=0,1,\ldots,N$, as
$$
  a_h(S_ku,v)=a_h(u,v) \qquad \forall v\in V_k,
$$
and the local edge operators
$S_{k l}:V_h\rightarrow V_{k l}$  for all $\Gamma_{k l} \subset \Gamma$, by
$$
  a_h(S_{k l}u,v)=a_h(u,v) \qquad \forall v\in V_{k l}.
$$
Analogous to the symmetric case, the additive Schwarz operator is then given by
$$
 S=S_0 + \sum_{k=1}^N S_k + \sum_{\Gamma_{k l}\subset \Gamma} S_{k l},
$$
and the variational equation (\ref{eq:FVdp}) is then replaced by the following equivalent system of equations
\begin{equation} \label{eq:sdp}
 S u_h^*=\hat{g},
\end{equation}
where $\hat{g}=\hat{g}_0 + \sum_{k=1}^N \hat{g}_k + \sum_{\Gamma_{k l}\subset \Gamma} \hat{g}_{k l}$, $\hat{g}_k=S_ku_h^*$ for $k=0,1,\ldots,N$, and $\hat{g}_{k l} = S_{k l}u_h^*$ for $\Gamma_{k l}\subset \Gamma$, with $u_h^*$ being the exact solution of (\ref{eq:sdp}).

\begin{theorem}\label{thm:ASM-edge-nsym}
There exists an $h_1$ such that, if $h\leq h_1$, then for any $u\in V_h$
$$ a(Su,Su)\lesssim a(u,u), \quad a(Su,u)\gtrsim \left(1+\log\left(\frac{H}{h}\right)\right)^{-2} a(u,u),
$$
where $H=\max_k(H_k)$ and $H_k=\mathrm{diam}(\Omega_k)$.
\end{theorem}
The proof of this theorem is a direct consequence of the abstract results of Section~\ref{sec:ASM-covol}, e.g. Theorem~\ref{thm:ASM-covol-nonsym},  and the propositions~\ref{prop:nonsym-part-h-bound},~\ref{C-S-spec-rad-bnd}, and~\ref{prop:low-bnd-sym} from Section~\ref{sec:techn-tools}.

Finally utilizing Theorem~\ref{thm:GMRES-estimate} we get a bound for the convergence of the GMRES method applied for solving (\ref{eq:sdp}).

%
%
%
\section{The abstract framework} \label{sec:ASM-covol}
we formulate an abstract framework for the convergence analysis of additive Schwarz methods accelerated by the GMRES iteration (cf. \cite{Saad:1986:GGM}) for a general finite volume element discretization.


We consider a family of finite dimensional subspaces $V^h$ indexed by the parameter $h$, an inner product $a(\cdot,\cdot)$ and its induced norm $\|\cdot\|_a:=\sqrt{a(\cdot,\cdot)}$, and a family of discrete problems: Find $u_h\in V^h$
$$
   a_h(u_h,v)=f(v) \quad \forall v \in V^h,
$$
where $a_h(u,v)$ is a nonsymmetric bilinear form. We assume that the nonsymmetric bilinear form is a small perturbation of the symmetric one, in the sense that, for all $h\leq h_0$ (a constant),
$$
  E_h(u,v):=a_h(u,v)-a(u,v)
$$
converges to zero as $h$ tends to zero satisfying the following uniform bound:
\begin{equation} \label{eq:nonsym-part-h-bound}
 \exists \; C_E >0: \; \forall h<h_0 \mbox{ and } \forall u, v \in V^h, \quad |E_h(u,v)|\leq C_Eh\|u\|_a\|v\|_a,
\end{equation}
where $C_E$ is a constant independent of $h$.

Let the space $V^h$ be decomposed into its subspaces as follows,
$$
   V^h=\sum_{k=0}^N V_k
$$
where $V_k \subset V^h$ for $k=0,\ldots,N$.

\subsection{Symmetric preconditioner}
For $k=0,\ldots,N$, we define the projection operator $T_k:V^h\rightarrow V_k$ as
\begin{eqnarray} \label{eq:def_T_k}
  a(T_k u,v)=a_h(u,v) \qquad \forall v \in V_k.
\end{eqnarray}
Note that the bilinear form $a(u,v)$ is an inner product in $V^h$, hence $T_k$ is a well defined linear operator. Let the additive Schwarz operator $T: V^h\rightarrow V^h$ be given as
\begin{eqnarray}
 T=\sum_{k=0}^N T_k,
\end{eqnarray}
and the original problem be replaced by
\begin{equation} \label{eq:AF_Tueqg}
  Tu_h=g,
\end{equation}

where $g=\sum_{k=0}^N g_k$ and $g_k=T_k u_h$.
Note that $T_k$ and $T$ are in general nonsymmetric. To solve this system we use the GMRES iteration, cf. \cite{Saad:1986:GGM}, whose convergence estimates  are based on the two parameters, the smallest eigenvalue of the symmetric part of the operator and the norm of the operator,
\begin{eqnarray} \label{eq:GMRES-abst-const}
  \beta_1=\inf_{u\not =0} \frac{a(Tu,u)}{\|u\|_a^2}, \qquad \beta_2=\sup_{u\not =0}   \frac{\|T u\|_a}{\|u\|_a }.
\end{eqnarray}
We state the classical theorem.
\begin{theorem} [Eisenstat-Elman-Schultz \cite{Eisenstat:1983:VIM})] \label{thm:GMRES-estimate}
If $\beta_1>0$, then the GMRES method for solving the linear system (\ref{eq:AF_Tueqg}) converges for any starting value $u_0\in V^h$ with the following estimate:
$$
 \|g-T u_m\|_a\leq\left(1-\frac{\beta_1^2}{\beta_2^2}\right)^{m/2}\|g-Tu_0\|_a,
$$
where $u_m$ is the m-th iterate of the GMRES method.
\end{theorem}

For any $u \in V^h$, we assume that there are functions $u_i\in V_i$, $i\ldots,N$, such that the following hold: There exists a positive constant $C_0$ (which may depend on the mesh parameters) such that
\begin{eqnarray}\label{eq:ASM-splitting}
  u = \sum_{k=0}^N u_k
\end{eqnarray}
and
\begin{eqnarray} \label{eq:ASM-split-est}
  \sum_{k=0}^N a(u_k,u_k) \leq C_0^2 a(u,u).
\end{eqnarray}

We also assume the following strengthened Cauchy-Schwarz inequality:
For any $k,l=1,\ldots,N$, let $\epsilon_{k l}$ be the minimal nonnegative constants such that
\begin{eqnarray} \label{eq:ass-CS-loc-ineq}
  a(u_k,u_l)\leq \epsilon_{k l}\|u_k\|_a \|u_l\|_a \qquad  u_k\in V_k, u_l\in V_l.
\end{eqnarray}
Let $\rho(\mathcal{E})$ be the spectral radius of the $N \times N$ symmetric matrix $\mathcal{E}=(\epsilon_{k l})_{k,l=1}^N$.

The following lemmas can be given.
\begin{lemma}
For any $M\in(1,2)$ there exists an $h_1\leq h_0$ such that if $h < h_1$ then the bilinear form $a_h(u,v)$ is uniformly bounded with respect to the inner product $a(u,v)$, i.e.
\begin{equation} \label{eq:CVform-bounded}
 \forall u, v \in V^h \quad |a_h(u,v)|\leq M\|u\|_a\|v\|_a.
\end{equation}
\end{lemma}
\begin{proof}
It follows from the assumption (\ref{eq:nonsym-part-h-bound}) that
\begin{eqnarray*}
 a_h(u,v) &=& a(u,v) + E_h(u,v) \\ &\leq& (1+C_Eh)\|u\|_a\|v\|_a
 \leq (1+C_Eh_1)\|u\|_a\|v\|_a.
\end{eqnarray*}
Taking $h_1=\min((M-1)/C_E,h_0)$ ends the proof.
\end{proof}
\begin{lemma}
For any $\alpha \in (0,1)$, there exists an $h_1\leq h_0$ such that if $h<h_1$ then the bilinear form $a_h(u,v)$ is uniformly $V^h$-elliptic in the $\parallel \cdot \parallel_a$-norm, i.e.
\begin{equation} \label{eq:CVform-Vh-elliptic}
 \forall u\in V^h \quad a_h(u,u)\geq \alpha\|u\|_a^2.
\end{equation}
\end{lemma}
\begin{proof}
By the assumption (\ref{eq:nonsym-part-h-bound}), we have
\begin{eqnarray*}
 a(u,u)\leq a_h(u,u) +|E_h(u,u)|\leq a_h(u,u)+C_Eha(u,u).
\end{eqnarray*}
If $h<h_1\leq h_0$ and $C_Eh_1\leq 1-\alpha$, then
$$
  a_h(u,u)\geq (1-C_Eh_1\,)a(u,u)\geq\: \alpha a(u,u),
$$
and the proof follows.
\end{proof}

We now state the main theorem of this section.
\begin{theorem} \label{thm:ASM-covol-symm}
There exists an $h_1\leq h_0$ such that if $h < h_1$ then
 \begin{eqnarray}
  a(Tu,Tu)\leq \beta_2^2 a(u,u), \label{eq:Tupper-bnd}\\
  a(Tu,u)\geq \beta_1 a(u,u), \label{eq:Tlower-bnd}
\end{eqnarray}
where $\beta_2=(2M(1+\rho(\mathcal{E})))$ and $\beta_1=(\alpha^2C_0^{-2} -\beta_2 C_E h)$.
\end{theorem}

\begin{remark} In some cases the constant $C_0$ in (\ref{eq:ASM-split-est})  may depend  on $h$ ($C_0=C_0(h))$ then $C_0(h)$ cannot grow  too fast with decreasing $h$, otherwise $\beta_1$ may become negative and our theory would not work, e.g.  if $\rho(\mathcal{E})$ is independent of $h$ which is usually the case in ASM methods, then it would be sufficient if
$
  \lim_{h\rightarrow 0} C_0^2(h)h =0,
$
because
then there exists an $h_1\leq h_0$ such that  $\beta_1$ is  positive for any
$h < h_1$.
\end{remark}

Before we prove Theorem~\ref{thm:ASM-covol-symm}, we need the following lemma.
\begin{lemma} \label{lem:sqtriangle}
Let  $u_k\in V_k$ for $k=0,\ldots,N$, then
$$
  \|\sum_{k=0}^N u_k\|_a^2\leq 2(1+\rho(\mathcal{E}))\sum_k\|u_k\|_a^2,
$$
where $\rho(\mathcal{E})$ is the spectral radius of the matrix $\mathcal{E}=(\epsilon_{k l})_{k,l=1}^N$.
\end{lemma}
\begin{proof}
We see that
\begin{eqnarray*}
  \|\sum_{k=0}^N u_k\|_a^2&\leq& 2\|u_0\|_a^2+2\|\sum_{k=1}^N u_k\|_a^2.
\end{eqnarray*}
Using (\ref{eq:ass-CS-loc-ineq}) and a Schwarz inequality in the $l_2$-norm we get

\begin{eqnarray*}
  \|\sum_{k=1}^N u_k\|_a^2&=&\sum_{k,l=1}^Na( u_k,u_l) \\
  &\leq& \sum_{k,l=1}^N  \epsilon_{k l}\|u_k\|_a\|u_l\|_a\\
  &\leq& \rho(\mathcal{E})\sqrt{ \sum_{k=1}^N  \|u_k\|_a^2}
  \sqrt{ \sum_{l=1}^N  \|u_l\|_a^2}\\&=&\rho(\mathcal{E}) \sum_{k=1}^N  \|u_k\|_a^2,
\end{eqnarray*}
and the proof follows.
\end{proof}

We now give the proof of Theorem \ref{thm:ASM-covol-symm}. It follows from Lemma~\ref{lem:sqtriangle} that
\begin{eqnarray*}
 a(Tu,Tu)&=&\|\sum_k T_ku,\sum_k T_k u\|_a\leq
2(1+\rho(\mathcal{E}))\sum_k\|T_ku\|_a^2.
\end{eqnarray*}
By (\ref{eq:def_T_k}) and (\ref{eq:CVform-bounded}), we get

\begin{eqnarray}
\sum_{k=0}^N\|T_ku\|_a^2=\sum_{k=0}^N a(T_k u,T_k u)
 &=& \sum_{k=0}^N  a_h(u,T_ku) \nonumber \\
 &=& a_h(u,T u) \label{eq:sumT_k-bound}\\
 &\leq& M \|u\|_a\|T u\|_a. \nonumber
\end{eqnarray}
The upper bound, cf. (\ref{eq:Tupper-bnd}), then follows with $\beta_2=(2M(1+\rho(\mathcal{E})))$.



To prove the lower bound, cf. (\ref{eq:Tlower-bnd}), we start with the splitting of $u\in V^h$, cf. (\ref{eq:ASM-splitting}), such that (\ref{eq:ASM-split-est}) holds.
Then using (\ref{eq:CVform-Vh-elliptic}), (\ref{eq:def_T_k}), a Schwarz inequality, and (\ref{eq:ASM-split-est}), we get

\begin{eqnarray*}
 \alpha a(u,u)
\leq a_h(u,u)
= \sum_{k=0}^N a_h(u,u_k)
&=&\sum_{k=0}^N a(T_k u,u_k)\\
&\leq& \sum_{k=0}^N \|T_k u\|_a \|u_k\|_a \\
&\leq & \sqrt{\sum_{k=0}^N \|T_k u\|_a^2} \; \sqrt{\sum_{k=0}^N\|u_k\|_a^2}\\
&\leq& C_0  \sqrt{\sum_{k=0}^N \|T_k u\|_a^2} \; \|u\|_a.
\end{eqnarray*}
This and (\ref{eq:sumT_k-bound}) then yield
\begin{eqnarray*}
 \alpha^2 a(u,u)
 &\leq&C_0^2 \sum_k\|T_k u\|_a^2= C_0^2a_h(u,Tu).
\end{eqnarray*}
Finally, from the assumption (\ref{eq:nonsym-part-h-bound}) and the upper bound (\ref{eq:Tupper-bnd}), we get
\begin{eqnarray*}
a_h(u,Tu) = a(u,Tu)+E_h(u,Tu) &\leq& a(u,Tu) + C_E h\|u\|_a\|Tu\|_a \\
&\leq& a(u,Tu) +\beta_2 C_E h\|u\|_a^2.
\end{eqnarray*}
Hence,
$$
  a(Tu,u) \geq (\alpha^2C_0^{-2} -\beta_2 C_E h)a(u,u).
$$
Taking $\beta_1=(\alpha^2C_0^{-2} -\beta_2 C_E h)$
we get the lower bound in (\ref{eq:Tlower-bnd}).

\subsection{Nonsymmetric preconditioner}
For $k=0,\ldots,N$, we define the projection operators $S_k: V^h\rightarrow V_k$ as
\begin{eqnarray} \label{eq:def_S_k}
  a_h(S_k u,v)=a_h(u,v) \qquad \forall v \in V_k.
\end{eqnarray}
Note that the bilinear form $a_h(u,v)$ is $V_k$-elliptic, cf. (\ref{eq:CVform-Vh-elliptic}), so $S_k$ is a well defined linear operator.
Now, introducing the additive Schwarz operator $S: V^h\rightarrow V^h$ as
\begin{eqnarray*}
 S=\sum_{k=0}^N S_k,
\end{eqnarray*}
we replace the original problem with
$$
  Su_h=g,
$$
where $g=\sum_{k=0}^N g_k$ and $g_k=S_k u_h$.

The main theorem of this section, in which we bound the constants from the estimate of the convergence speed of GMRES, cf. (\ref{eq:GMRES-abst-const}) and Theorem~\ref{thm:GMRES-estimate}, is the following:
\begin{theorem} \label{thm:ASM-covol-nonsym}
There exists $h_1<h_0$ such that for any $h<h_1$, the following bounds hold.
 \begin{eqnarray*}
  a(Su,Su)\leq \gamma_2^2 a(u,u),\\ 
  a(Su,u)\geq \gamma_1 a(u,u)
\end{eqnarray*}
where
$\gamma_2=\frac{2M}{\alpha}(1+\rho(\mathcal{E}))$ and
$\gamma_1=
\frac{\alpha^3}{M^2 C_0^2} - \gamma_2C_E h$, and, as before, $\rho(\mathcal{E})$ is the spectral radius of the matrix $\mathcal{E}=(\epsilon_{k l})_{k,l}^N$.
\end{theorem}
\begin{proof}
We follow the lines of proof of Theorem~\ref{thm:ASM-covol-symm}.
For the upper bound, we use Lemma~\ref{lem:sqtriangle} to see that
\begin{eqnarray*}
 a(Su,Su)&\leq&
2(1+\rho(\mathcal{E}))\sum_k\|S_ku\|_a^2.
\end{eqnarray*}
Using (\ref{eq:CVform-Vh-elliptic}), (\ref{eq:def_S_k}), and (\ref{eq:CVform-bounded}), we get
\begin{eqnarray}
\alpha \sum_{k=0}^N a(S_k u,S_k u) \leq
\sum_{k=0}^N a_h(S_k u,S_k u)
 &=& \sum_{k=0}^N  a_h(u,S_ku) \nonumber \\
 &=& a_h(u,Su) \label{eq:sumS_k-bound} \\
 &\leq& M \|u\|_a\|S u\|_a. \nonumber
\end{eqnarray}
And, the upper bound is proved with $\gamma_2=(\frac{2M}{\alpha}(1+\rho(\mathcal{E})))$.

For the lower bound, again, we use the splitting (\ref{eq:ASM-splitting}) of $u\in V^h$ such that (\ref{eq:ASM-split-est}) holds.
Next (\ref{eq:CVform-Vh-elliptic}), (\ref{eq:ASM-splitting}), (\ref{eq:def_S_k}), (\ref{eq:CVform-bounded}), a Schwarz inequality in $l_2$, and (\ref{eq:ASM-split-est}) yield that
\begin{eqnarray*}
 \alpha a(u,u) \leq a_h(u,u) = \sum_k a_h(u,u_k) &=& \sum_k a_h(S_k u,u_k) \nonumber\\
 &\leq& M\sum_k \|S_k u\|_a \|u_k\|_a \nonumber\\
 &\leq& M\sqrt{\sum_k\|S_k u\|_a^2} \sqrt{\sum_k\|u_k\|_a^2} \nonumber\\
 &\leq& M \: C_0\|u\|_a \sqrt{\sum_k\|S_k u\|_a^2} .
\end{eqnarray*}
Combining the estimate above with (\ref{eq:sumS_k-bound}), we get
\begin{eqnarray*}
 \alpha^2 a(u,u)
&\leq&M^2C_0^2 \sum_k\|S_k u\|_a^2\leq  \frac{M^2\:C_0^2}{\alpha}a_h(u,Su).
\end{eqnarray*}
Finally, using similar arguments as in the proof of Theorem~\ref{thm:ASM-covol-symm}, we can conclude that
$$
  a(u,S u)\geq (\frac{\alpha^3}{M^2 C_0^2} - \gamma_2C_E h)a(u,u),
$$
for any $h\leq h_1$.
\end{proof}

%
%
%
\section{Technical tools}\label{sec:techn-tools}
In this section, we present the technical results necessary for the proof of Theorem~\ref{thm:ASM-edge-sym}. We use the abstract framework introduced in the previous section, for which we verify the assumption (\ref{eq:nonsym-part-h-bound}), show that $\rho({\cal E})$ is bounded by a constant, and finally give an estimate for the $C_0^2$ such that (\ref{eq:ASM-splitting})-(\ref{eq:ASM-split-est}) to hold, all formulated as propositions.

We start with the proposition which shows that (\ref{eq:nonsym-part-h-bound}) holds true for the two bilinear forms $a(\cdot,\cdot)$ and $a_h(\cdot,\cdot)$ of (\ref{eq:sym-bil-form}) and (\ref{eq:FV-bil-form}), respectively.
\begin{proposition}\label{prop:nonsym-part-h-bound}
It holds that
$$
\exists \; C_E >0: \; \forall u, v \in V^h, \quad |a_h(u,v)-a(u,v)|\leq C_Eh\|u\|_a\|v\|_a,
$$
where $C_E$ is a constant independent of $h$ and the jumps of the coefficients across $\partial D_j$s, but may depend on $C_\Omega$ in  (\ref{eq:subd_D_K}) .
\end{proposition}
The proof follow the same lines of proof of Lemma~3.1 in \cite{Ewing:2000:MFV}, cf. also \cite{Ewing:2002:OTA}.

Next, we present three known lemmas. The first lemma is the so-called Sobolev like inequality, cf. e.g. Lemma~7 in \cite{Smith:1996:DDP}.
\begin{lemma}[Discrete Sobolev like inequality] \label{lem:Sobolev-ineq-disc}
Let $u\in V^h(\Omega_k)$, then
$$
 \|u\|_{L^{\infty}(\Omega_k)}^2 \lesssim
 \left( 1+  \log\left(\frac{H_k}{h}\right)\right)
 \left(H_k^{-2}\|u\|_{L^2(\Omega_k)}^2+|u|_{H^1(\Omega_k)}^2\right)
$$
where $H_k=\mathrm{diam}(\Omega_k)$.
\end{lemma}
The second lemma is the well known extension theorem for discrete harmonic functions, cf. e.g. Lemma~5.1 in \cite{Bjorstad:1986:IMS}.
\begin{lemma}[Discrete extension theorem] \label{lem:extension-thm}
Let $u\in W_k$, then
$$
|u|_{H^1(\Omega_k)} \lesssim |u|_{H^{1/2}(\partial\Omega_k)}.
$$
\end{lemma}
Finally, the third lemma gives an estimate of the $H^{1/2}_{00}(\Gamma_{kl})$ norm of a finite element function which is zero on $\partial\Omega_k\setminus\Gamma_{kl}$ by its $H^{1/2}$ seminorm and $L^\infty$ norm, cf. e.g. Lemma~4.1 in \cite{Tallec:1998:NND}.
\begin{lemma} \label{lem:edge-Hh00-est}
 Let $u\in W_k$ such that $u_{|\partial\Omega_k \setminus \Gamma_{k l}}=0$, then
$$
  \|u\|_{H^{1/2}_{00}(\Gamma_{k l})}^2\lesssim
    |u|_{H^{1/2}(\Gamma_{k l})}^2 +
   \left( 1+  \log\left(\frac{H_k}{h}\right)\right)\|u\|_{L^\infty(\Gamma_{k l})}^2
$$
\end{lemma}

In the following we present additional set of technical lemmas. The first one is a simple result which will be useful to estimate the $H^1$ seminorm of functions from the coarse space $V_0$.
\begin{lemma} \label{lem:inex-bil-form-est}
For $u\in V_0$ and $C$ being an arbitrary constant, the following holds, i.e.
$$
   |u|_{H^1(\Omega_k)}^2\lesssim \sum_{x\in \mathcal{V}_k} |u(x)-C|^2,
$$
where $\mathcal{V}_k$ is the set of all vertices of $\Omega_k$ which are not on $\partial\Omega$.
\end{lemma}
\begin{proof}
Note that $u_{|\Omega_k}=\sum_{x\in  \mathcal{V}_k}u(x){\phi_x}_{|\Omega_k}$, where $\phi_x$ is a discrete harmonic function which is equal to one at $x$, zero at $\mathcal{V}_k\setminus\{x\}$, and linear along the edges $\Gamma_{kl} \subset \partial\Omega_k$ .
Thus, for any constant $C$, we have
\begin{eqnarray*}
|u|_{H^1(\Omega_k)}^2=|u-C|_{H^1(\Omega_k)}^2\lesssim
\sum_{x\in  \mathcal{V}_k}|u(x)-C|^2|\phi_x|_{H^1(\Omega_k)}^2\lesssim
\sum_{x\in  \mathcal{V}_k}|u(x)-C|^2.
\end{eqnarray*}
The last inequality follows from the standard estimate of $H^1$ seminorm of a coarse nodal function, and the fact that a discrete harmonic function has the minimal energy of all functions taking the same values on the boundary.
\end{proof}

\begin{definition}\label{def:coarse-interp}
 Let $I_H:V_h\rightarrow V_0$ be a coarse interpolant defined by the values of $u$ at the vertices $\mathcal{V}$, i.e. let $I_H u \in V_0$
and $I_H u(x)=u(x)$ for $x\in   \mathcal{V}$.
\end{definition}


\begin{lemma}\label{lem:coarse-loc-est}
For any $u\in V_h$, the following holds, i.e.
\begin{eqnarray*}
  |I_H u|_{H^1(\Omega_k)}^2\lesssim \left( 1+  \log\left(\frac{H_k}{h}\right)\right)
  |u|_{H^1(\Omega_k)}^2.
\end{eqnarray*}
\end{lemma}
\begin{proof}
From lemmas~\ref{lem:inex-bil-form-est} and \ref{lem:Sobolev-ineq-disc}, we get
$$ |I_H u|_{H^1(\Omega_k)}^2\lesssim
\left( 1+  \log\left(\frac{H_k}{h}\right)\right)
\left(H_k^{-2}\|u-C\|_{L^2(\Omega_k)}^2+|u-C|_{H^1(\Omega_k)}^2\right),
$$
for any constant $C$. A scaling argument and a quotient space argument
complete the proof.
\end{proof}

\begin{lemma}\label{lem:proof-edge-lem}
Let $\Gamma_{k l}\subset\partial\Omega_k$ be an edge, and
$u_{k l}\in W_k$ be a function defined as $u_{k l}(x)=u(x)-I_Hu(x)$ on $\Gamma_{k l}$, and as zero on $\partial\Omega_k\setminus\Gamma_{k l}$ for any $u\in V_h$. Then, the following holds, i.e.
\begin{eqnarray}
  |u_{k l}|_{H^{1/2}_{00}(\Gamma_{k l})}^2\lesssim \left( 1+  \log\left(\frac{H_k}{h}\right)\right)^2
  |u|_{H^1(\Omega_k)}^2
\end{eqnarray}
\end{lemma}
\begin{proof}
By Lemma~\ref{lem:edge-Hh00-est}, we get
\begin{equation} \label{eq:proof-edge-lem:1}
 \|u_{k l}\|_{H^{1/2}_{00}(\Gamma_{k l})}^2\lesssim
    |u-I_H u|_{H^{1/2}(\Gamma_{k l})}^2 +
   \left( 1+  \log\left(\frac{H_k}{h}\right)\right)\|u-I_H u\|_{L^\infty(\Gamma_{k l})}^2.
\end{equation}
The first term can be estimated using the standard trace theorem, a triangle inequality and Lemma~\ref{lem:coarse-loc-est} as follows,
\begin{equation} \label{eq:proof-edge-lem:2}
   |u-I_H u|_{H^{1/2}(\Gamma_{k l})}^2\lesssim
   |u|_{H^1(\Omega_k)}^2 +  |I_Hu|_{H^1(\Omega_k)}^2 \lesssim
\left( 1+  \log\left(\frac{H_k}{h}\right)\right)
  |u|_{H^1(\Omega_k)}^2.
\end{equation}
For any constant $C$, we note that $u-I_Hu=u-C -I_H(u-C)$ on $\Gamma_{k l}$, and since $I_H u$ is a linear function along $\Gamma_{k l}$,
$\|I_H u\|_{L^\infty(\Gamma_{k l})}\leq \|u\|_{L^\infty(\Gamma_{k l})}$.
Hence the $L^\infty$ norm of $u-I_H u$ in (\ref{eq:proof-edge-lem:1}) can be estimated as follows,
\begin{eqnarray*}
 \|u-I_H u\|_{L^\infty(\Gamma_{k l})}^2&\lesssim&
 \|u-C\|_{L^\infty(\Gamma_{k l})}^2 + \|I_H (u-C)\|_{L^\infty(\Gamma_{k l})}^2 \\
&\leq& \|u-C\|_{L^\infty(\Gamma_{k l})}^2\\
&\lesssim&
\left( 1+  \log\left(\frac{H_k}{h}\right)\right)
\left(H_k^{-2}\|u-C\|_{L^2(\Omega_k)}^2+|u-C|_{H^1(\Omega_k)}^2\right),
\end{eqnarray*}
where $C$ as any arbitrary constant. The last inequality is due to Lemma~\ref{lem:Sobolev-ineq-disc}.
Finally, a scaling argument and a quotient space argument yield
\begin{eqnarray*}
  \|u-I_H u\|_{L^\infty(\Gamma_{k l})}^2\lesssim
 \left( 1+  \log\left(\frac{H_k}{h}\right)\right)
|u|_{H^1(\Omega_k)}^2.
\end{eqnarray*}
The above estimate together with the estimates (\ref{eq:proof-edge-lem:2}) and (\ref{eq:proof-edge-lem:1}), complete the proof.
\end{proof}

A standard coloring argument bounds the spectral radius, and is given here in our second proposition.
\begin{proposition}\label{C-S-spec-rad-bnd}
Let $\mathcal{E}$ be the symmetric matrix of Cauchy-Schwarz coefficients, cf. (\ref{eq:ass-CS-loc-ineq}), for the subspaces $V_k$, $V_l$, and $V_{kl}$, $k,l=1,\ldots,N$, of the decomposition (\ref{eq:space-decomp}). Then,
$$
   \rho(\mathcal{E}) \leq C,
$$
where $C$ is a positive constant independent of the coefficients and mesh parameters.
\end{proposition}

The third and final proposition gives an estimate of the $C_0^2$ such that (\ref{eq:ASM-splitting})-(\ref{eq:ASM-split-est}) hold for any $u\in V_h$.
\begin{proposition}\label{prop:low-bnd-sym}
 For any $u\in V_h$ there exists $u_k\in V_k$ $k=0,1,\ldots,N$ and $u_{k l}\in V_{k l}$ such that
\begin{eqnarray*}
 u = u_0 + \sum_k u_k + \sum_{\Gamma_{k l}\subset \Gamma}u_{k l}
\end{eqnarray*}
and
\begin{eqnarray*}
 a(u_0,u_0)+ \sum_k a(u_k,u_k) + \sum_{\Gamma_{k l}\subset \Gamma}a(u_{k l},u_{k l})&\lesssim& \left( 1+  \log\left(\frac{H}{h}\right)\right)^2
a(u,u),
\end{eqnarray*}
where $H=\max_k H_k$ with $H_k=\mathrm{diam}(\Omega_k)$.
\end{proposition}
\begin{proof}
We first set $u_0=I_H u\in V_0$, cf. Definition~\ref{def:coarse-interp}.
Next, let $u_k\in V_k$ for $k=1,\ldots,N$, be defined as $\mathcal{P}_k u_{|\overline{\Omega}_k}$ on $\Omega_k$, be extended by zero to the rest of $\Omega$.

Now define $w = u-u_0-\sum_k u_k$.
Note that $w$ is discrete harmonic inside each subdomain $\Omega_k$, since $u_0$ is discrete harmonic in the same way, and the sum
$$
   (w+u_0)_{|\overline{\Omega}_k}= u_{|\overline{\Omega}_k} - \mathcal{P}_k u_{|\overline{\Omega}_k} = \mathcal{H}_k u_{|\overline{\Omega}_k}
$$
is in fact a function of $W_k$. Moreover,
$$
  w(x) = u(x)-I_Hu(x)=0 \quad x\in \mathcal{V}.
$$
Consequently, $w$ can be decomposed as follows,
$$
  w=\sum_{\Gamma_{k l}\subset \Gamma} u_{k l},
$$
where $u_{k l}\in V_{k l}$, with $u_{|\Gamma_{k l}} = w_{|\Gamma_{k l}}$.

We now prove the inequality by considering each term at a time.
For the first term, by Lemma~\ref{lem:coarse-loc-est}, we see that
\begin{eqnarray}
 a(u_0,u_0)&\lesssim &\sum_k \Lambda_k |u_0|_{H^1(\Omega_k)}^2
 = \sum_k \Lambda_k |I_H u|_{H^1(\Omega_k)}^2   \label{eq:proof-lower-bnd-coarse}\\
\nonumber
 &\lesssim &\left(1+\log\left(\frac{H}{h}\right)\right) \sum_k a_k(u,u)=
   \left(1+\log\left(\frac{H}{h}\right)\right)  a(u,u).
\end{eqnarray}
For the second term, since $\mathcal{P}_k$ is the orthogonal projection in $a_k(u,v)$, we get
\begin{eqnarray} \label{eq:proof-lower-bnd-int}
 \sum_ka(u_k,u_k)=\sum_ka_k(\mathcal{P}_ku_{|\overline{\Omega}_k},\mathcal{P}_ku_{|\overline{\Omega}_k})\leq \sum_ka_k(u_{|\overline{\Omega}_k},u_{|\overline{\Omega}_k})=a(u,u).
\end{eqnarray}
And, for the last term, let $\Gamma_{k l}\subset \Gamma$ be the edge which is common to both $\Omega_k$ and $\Omega_l$. Note that $u_{k l} \in V_{k l}$ has support both in $\overline{\Omega}_k\cup\overline{\Omega}_l$.
By Lemma~\ref{lem:extension-thm}, we note that
\begin{eqnarray*}
 a(u_{k l},u_{k l}) = \sum_{s=k,l} a_s(u_{k l},u_{k l})\lesssim
  \sum_{s=k,l} \Lambda_s |u_{k l}|_{H^1(\Omega_s)}^2 \lesssim
  ( \sum_{s=k,l} \Lambda_s)|u_{k l}|_{H^{1/2}_{00}(\Gamma_{k l})}^2.
\end{eqnarray*}
Utilizing Lemma~\ref{lem:proof-edge-lem} for $s=k$ if $\Lambda_k\geq \Lambda_l$ (otherwise we take $s=l$), and (\ref{eq:coeff-smooth-subd}), we get
\begin{eqnarray*}
  ( \sum_{s=k,l} \Lambda_s) |u_{k l}|_{H^{1/2}_{00}(\Gamma_{k l})}^2
  &\lesssim& \left( 1+  \log\left(\frac{H}{h}\right)\right)^2
  \Lambda_k|u|_{H^1(\Omega_k)}^2 \\
  &\lesssim&
  \left( 1+  \log\left(\frac{H}{h}\right)\right)^2
  a_k(u,u).
\end{eqnarray*}
Combining the last two estimates, we get
\begin{eqnarray}  \label{eq:proof-lower-bnd-edge}
 \sum_{\Gamma_{k l}\subset \Gamma}a(u_{k l},u_{k l}) &\lesssim&
 \sum_{\Gamma_{k l}\subset \Gamma} \sum_{s=k,l}  \left( 1+  \log\left(\frac{H}{h}\right)\right)^2
 a_s(u,u)
\\
& \lesssim&  \left( 1+  \log\left(\frac{H}{h}\right)\right)^2 a(u,u). \nonumber
\end{eqnarray}
The proof then follows by summing (\ref{eq:proof-lower-bnd-coarse}), (\ref{eq:proof-lower-bnd-int}), and (\ref{eq:proof-lower-bnd-edge}) together.
\end{proof}

%
%
%
\section{Numerical experiments} \label{sec:numer-tests}
In this section we present some numerical test cases showing the performance of the proposed method. All of the following results have been obtained using Matlab by employing a modified GMRES method where the standard $l_2$ inner product have been replaced with the $a(\cdot,\cdot)$ inner product. The GMRES method is then accelerated with our preconditioners and then run until the $l_2$ norm of the initial residual is reduced by a factor of $10^{6}$, i.e., until $\|r_i\|_2/\|r_0\|_2\leq 10^{-6}$.

We consider the domain $\Omega=[0,1]\times[0,1]$ and subdivide it into equal square subdomains with coarse mesh parameter $H$ and fine mesh parameter $h$. The right hand side $f$ is chosen as $1$.
\begin{table}[htb]
\centering
\begin{tabular}{|c|cccccc|}\hline
h/H&$\frac{1}{4}$&$\frac{1}{8}$&$\frac{1}{16}$&$\frac{1}{32}$&$\frac{1}{64}$&$\frac{1}{128}$\\ \hline
$\frac{1}{8}$&7 (5.80e-1)&&&&&\\
$\frac{1}{16}$&9 (3.72e-1)&10 (5.60e-1)&&&&\\
$\frac{1}{32}$&11 (2.48e-1)&13 (3.57e-1)&10 (5.56e-1)&&&\\
$\frac{1}{64}$&13 (1.76e-1)&16 (2.41e-1)&14 (3.53e-1)&10 (5.56e-1)&&\\
$\frac{1}{128}$&15 (1.30e-1)&19 (1.72e-1)&17 (2.38e-1)&13 (3.53e-1)&10 (5.55e-1)&\\
$\frac{1}{256}$&16 (1.01e-1)&21 (1.28e-1)&20 (1.70e-1)&16 (2.38e-1)&13 (3.52e-1)&10 (5.54e-1)\\
\hline
\end{tabular}
\caption{Iteration numbers and estimates of the smallest eigenvalue (in parentheses) for the symmetric preconditioner for increasing values of $h$ and $H$. Here $A=2+\sin(\pi x)\sin(\pi y)$.}
\label{tbl:numres1}
\end{table}
The number of iterations and estimates of the smallest eigenvalue of the symmetric part of the preconditioned operator $T$, i.e., the smallest eigenvalue of $\frac{1}{2}\left(T^t+T\right)$, are presented in the tables below for each of the problems under consideration.
Our numerical results have shown that the second parameter, i.e. the norm of the operator, which is used in describing the convergence rate of the GMRES\- iteration is a constant independent of the mesh parameters and the co\-efficient $A$, which is in agreement with our analysis. 
\begin{table}[htb]
\centering
\begin{tabular}{|c|cccccc|}\hline
h/H&$\frac{1}{4}$&$\frac{1}{8}$&$\frac{1}{16}$&$\frac{1}{32}$&$\frac{1}{64}$&$\frac{1}{128}$\\ \hline
$\frac{1}{8}$&10 (5.31e-1)&&&&&\\
$\frac{1}{16}$&12 (3.07e-1)&13 (4.31e-1)&&&&\\
$\frac{1}{32}$&14 (1.77e-1)&18 (2.42e-1)&14 (4.36e-1)&&&\\
$\frac{1}{64}$&15 (1.21e-1)&23 (1.61e-1)&18 (2.82e-1)&12 (5.20e-1)&&\\
$\frac{1}{128}$&17 (8.93e-2)&27 (1.17e-1)&22 (1.94e-1)&16 (3.37e-1)&11 (5.53e-1)&\\
$\frac{1}{256}$&20 (6.94e-2)&31 (8.90e-2)&26 (1.41e-1)&20 (2.28e-1)&14 (3.57e-1)&11 (5.57e-1)\\
\hline
\end{tabular}
\caption{Iteration numbers and estimates of the smallest eigenvalue (in parentheses) for the symmetric preconditioner for increasing values of $h$ and $H$. Here $A=2+\sin(10\pi x)\sin(10\pi y)$.}
\label{tbl:numres2}
\end{table}

For the first numerical experiment we test the dependency of the iteration number and the smallest eigenvalue on the mesh parameters $h$ and $H$ when the coefficient $A$ is equal to $2+\sin(\pi x)\sin(\pi y)$, and report the results in Table~\ref{tbl:numres1}. We observe that the number of iteration required to converge increases and the smallest eigenvalue decreases as $\frac{H}{h}$ increases, however, the changes happen very slowly suggesting a poly-logarithmic dependence as predicted in our theory.

In the following two numerical experiments, we perform the same type of experiments as the previous one, for both the symmetric and the nonsymmetric variant of the preconditioner, and $A$ equals to $2+\sin(10\pi x)\sin(10\pi y)$. The results are reported in Table~\ref{tbl:numres2} and \ref{tbl:numres3}, respectively. We observe a convergence behavior which is similar to the one in the first experiment, once again confirming our analysis. We note also that the performances of the two variants of the preconditioner are almost identical.

\begin{table}[htb]

\centering
\begin{tabular}{|c|cccccc|}\hline
h/H&$\frac{1}{4}$&$\frac{1}{8}$&$\frac{1}{16}$&$\frac{1}{32}$&$\frac{1}{64}$&$\frac{1}{128}$\\ \hline
$\frac{1}{8}$&10 (5.20e-1)&&&&&\\
$\frac{1}{16}$&12 (3.11e-1)&13 (4.25e-1)&&&&\\
$\frac{1}{32}$&14 (1.79e-1)&18 (2.43e-1)&14 (4.44e-1)&&&\\
$\frac{1}{64}$&15 (1.21e-1)&23 (1.62e-1)&18 (2.84e-1)&12 (5.25e-1)&&\\
$\frac{1}{128}$&17 (8.94e-2)&27 (1.17e-1)&22 (1.95e-1)&16 (3.38e-1)&11 (5.54e-1)&\\
$\frac{1}{256}$&20 (6.94e-2)&31 (8.90e-2)&26 (1.41e-1)&20 (2.28e-1)&14 (3.57e-1)&11 (5.57e-1)\\
\hline
\end{tabular}
\caption{Iteration numbers and estimates of the smallest eigenvalue (in parentheses) for the nonsymmetric preconditioner for increasing values of $h$ and $H$. Here $A=2+\sin(10\pi x)\sin(10\pi y)$.}
\label{tbl:numres3}
\end{table}

\begin{figure}[htb]
\centering
         \includegraphics[scale=0.20]{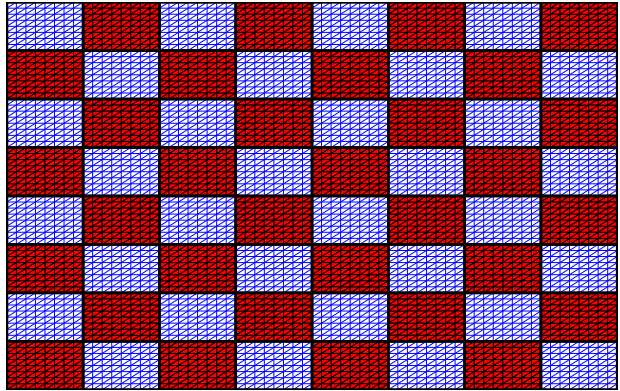}
  \caption{Checkerboard distribution of $A$, with $A = \alpha_1(2+\sin(10\pi x)\sin(10\pi y))$, where $\alpha_1 = \hat\alpha_1$ in the red (shaded) subdomains and 1 otherwise.}
  \label{fig:checkerboard}
\end{figure}

\begin{table}[htb]
\centering
\begin{tabular}{|c|c c|}
\hline
$\hat\alpha_1$ &Symmetric variant & Nonsymmetric variant\\
\hline
  $10^0$&23 (1.61e-1)&23 (1.62e-1)\\
  $10^1$&26 (1.61e-1)&26 (1.61e-1)\\
  $10^2$&27 (1.60e-1)&27 (1.60e-1)\\
  $10^3$&27 (1.60e-1)&27 (1.60e-1)\\
  $10^4$&27 (1.60e-1)&27 (1.60e-1)\\
  $10^5$&27 (1.60e-1)&27 (1.60e-1)\\
  $10^6$&27 (1.60e-1)&27 (1.60e-1)\\
\hline
\end{tabular}
\vspace{5mm}
\caption{Iteration numbers and estimates of the smallest eigenvalue for different values of $\alpha_1$ in the coefficient $A=\alpha_1(2+\sin(10\pi x)\sin(10\pi y))$ and a fixed mesh $h=1/64$ and $H=1/8$.}
\label{tbl:checkerboard}
\end{table}
In the last example we consider an example where $A$ is discontinuous across subdomains, given as $A=\alpha_1(2+\sin(10\pi x)\sin(10\pi y))$ with $\alpha_1$ being a constant in each subdomain. We divide $\Omega$ into equal square subdomains with diameter $H=1/8$ and let the fine triangulation have mesh size $h=1/64$. We then assign the parameter $\alpha_1$ in the coefficient $A$, the value $1$ (white subdomain) or the value $\hat\alpha_1$ (red or shaded subdomain) in a checkerboard fashion as depicted in Figure~\ref{fig:checkerboard}. Number of iterations required to converge and estimates of the smallest eigenvalues for different values of $\alpha_1$ (varying jumps) are reported in Table~\ref{tbl:checkerboard}, showing that the convergence is independent of the jumps in the coefficient supporting our analysis. Again, we see an identical performance of the two variants of the algorithm.





\end{document}